\documentclass[11pt]{amsart}
\usepackage{amsmath,amssymb,url}
\newtheorem{theorem}{Theorem}[section]
\newtheorem{conjecture}[theorem]{Conjecture}
\newtheorem{corollary}[theorem]{Corollary}
\newtheorem{definition}[theorem]{Definition}
\newtheorem{lemma}[theorem]{Lemma}
\newtheorem{proposition}[theorem]{Proposition}

\theoremstyle{remark}
\newtheorem{example}[theorem]{Example}
\newtheorem{remark}[theorem]{Remark}

\newcommand{\vanish}[1]{}\parskip=12pt

\newcommand{\0}{\widehat{0}}

\newcommand{\SF}{\mbox{SF}}

\begin{document}
\title[Strongly Bernoulli type truncation games on words]{Enumeration by kernel
  positions for strongly Bernoulli type truncation games on words} 
\author{G\'abor Hetyei}
\address{Department of Mathematics and Statistics, UNC Charlotte, 
	Charlotte, NC 28223}
\email{ghetyei@uncc.edu}
\thanks{This work was completed while the author was on reassignment of
   duties sponsored by the
   University of North Carolina at Charlotte.}  

\begin{abstract}
We find the winning strategy for a class of truncation
games played on words. As a consequence of the present author's recent
results on some of these games we obtain new formulas for Bernoulli
numbers and polynomials of the second kind and a new combinatorial model
for the number of connected permutations of given rank. For connected
permutations, the decomposition used to find the winning strategy is
shown to be bijectively equivalent to King's decomposition, used to 
recursively generate a transposition Gray code of the connected
permutations. 
\end{abstract}

\maketitle

\section*{Introduction}
In a recent paper~\cite{Hetyei-EKP} the present author introduced a
class of progressively finite games played on ranked posets, 
where each move of the winning strategy is unique and the positions
satisfy the following uniformity criterion:  each
position of a given rank may be reached from the same number of
positions of a given higher rank in a single move. As a consequence, the
kernel positions 
of a given rank may be counted by subtracting from the number of all
positions the appropriate multiples of the kernel positions of lower
ranks. The main example in~\cite{Hetyei-EKP} is the {\em original
  Bernoulli game}, a truncation game played on pairs of words of the
same length, for which the number of kernel positions of rank $n$ is
a signed factorial multiple of the Bernoulli number of the
second kind $b_n$. Similarly to this game, most examples mentioned
in~\cite{Hetyei-EKP} are also truncation games played on words, where
the partial order is defined by taking initial segments and the rank is
determined by the length of the words involved. 

In this paper we consider a class of {\em strongly Bernoulli type
  truncation games} played on words, for which
  we do not require the uniformity condition on the rank 
to be satisfied. We show that for such games, the winning 
strategy may be found by decomposing each kernel position as a
concatenation of {\em elementary kernel factors}. This decomposition is
unique. All truncation games considered in~\cite{Hetyei-EKP} (including
the ones played on pairs or triplets of words) are isomorphic to a
strongly Bernoulli type truncation game. For most of these examples,
the elementary kernel factors of a given type are also easy to
enumerate. Thus we may obtain explicit summation formulas and
  non-alternating recurrences for numbers which were expressed
  in~\cite{Hetyei-EKP} as coefficients in a generating function or by
  alternating recurrences. The explicit 
summation formulas are obtained by considering the entire unique
decomposition of each kernel position, the non-alternating recurrence
is obtained  by considering the removal of the last elementary kernel
factor only. Thus we find some new identities for the Bernoulli
polynomials and numbers of the second kind, and shed new light on
King's~\cite{King} decomposition of ``indecomposable'' permutations.   

The paper is structured as follows. After the Preliminaries, the main
unique decomposition theorem is stated in Section~\ref{s_Btg}. 
In the subsequent sections we consider games to which this result is
applicable: we show they are isomorphic to strongly Bernoulli type 
truncation games, we find formulas expressing their elementary kernel
factors of a given type, and use these formulas to express the number of
kernel positions as an explicit sum and by a non-alternating
recurrence. Most detail is given for the original Bernoulli game in
Section~\ref{s_ob2}, omitted details in other sections are replaced by
references to the appropriate part of this section. As a consequence of our
analysis of the original Bernoulli game, we obtain an explicit summation
formula of the Bernoulli numbers of the second kind, expressing them as a sum
of entries of the same sign. We also obtain a non-alternating recurrence
for their absolute values. 

In Section~\ref{s_MR} we consider a
restriction of the original Bernoulli game to a set of positions, where
the kernel positions are identifiable with the {\em connected} or {\em
  indecomposable} permutations forming an algebra basis of the
Malvenuto-Reutenauer Hopf algebra~\cite{Malvenuto-Reutenauer}. For these
the recurrence obtained by the removal of the last elementary kernel
factor is numerically identical to the recurrence that may be found in 
King's~\cite{King} recursive construction of a transposition Gray code
for the connected permutations. We show that this is not a coincidence:
there is a bijection on the set of permutations, modulo
which King's recursive step corresponds to the removal of the last
elementary kernel factor in the associated {\em place-based 
  non-inversion tables} (a variant of the usual inversion tables). Our
result inspires another systematic algorithm to list all connected
permutations of a given order, and a new combinatorial model for the
numbers of connected permutations of order $n$, in which this number 
arises as the total weight of all permutations
of order $n-2$, such that the highest weight is associated to the
permutations having the most {\em strong fixed points} (being thus
the ``least connected'').    

Section~\ref{s_pb2} contains the consequences of our main result to
Bernoulli polynomials of the second kind. Here we observe that we obtain
the coefficients of these polynomials when we expand them in the basis
$\{\binom{x+1}{n}\::\: n\geq 0\}$, and obtain a new formula for the
Bernoulli numbers of the second kind. 

Finally, in Section~\ref{s_fB} we
consider the {\em flat Bernoulli game}, whose kernel positions have the
generating function $t/((1-t)(1-\ln(1-t))$ and conclude the section with
an intriguing conjecture that for a long random initial word a novice player
could not decrease the chance of winning below $50\%$ by simply removing
the last letter in the first move. 

\section{Preliminaries}

\subsection{Progressively finite games}

A progressively finite two-player game is a  game whose positions may be
represented by the vertices of a directed graph that contains no directed
cycle nor infinite path, the edges represent valid moves. Thus
the game always ends after a finite number of moves. 
The players take alternate turns to move along a directed edge to 
a next position, until one of them reaches a {\em winning position} with
no edge going out: the player who moves into this position is declared a
winner, the next player is unable the move. 

The winning strategy for a progressively finite game may be found by
calculating the {\em Grundy number} (or Sprague-Grundy number) of each
position, the method is well-known, a sample
reference is~\cite[Chapter 11]{Tucker}. The positions with Grundy number
zero are called {\em kernel positions}.  A player has a winning strategy
exactly when he or she is allowed to start from a non-kernel position. 
All games considered in this paper are progressively finite.

\subsection{The original Bernoulli game and its generalizations}
\label{s_b2}

In~\cite{Hetyei-EKP} the present author introduced the {\em original
 Bernoulli game} as the following progressively finite two-player game. 
The positions of rank $n>0$ in the game are all pairs of words
 $(u_1\cdots u_n,v_1\cdots v_n)$ such that 
\begin{itemize}
\item[(i)] the letters $u_1,\ldots,u_n$ and $v_1,\ldots,v_n$ are
  positive integers; 
\item[(ii)] for each $i\geq 1$ we have $1\leq u_i, v_i\leq i$.
\end{itemize}
A valid move consists of replacing the pair $(u_1\cdots
u_n,v_1\cdots v_n)$ with $(u_1\cdots u_m,v_1\cdots v_m)$ for
some $m\geq 1$ satisfying $u_{m+1}\leq v_j$ for $j=m+1,\ldots, n$. 
The name of the game refers to the following fact~\cite[Theorem
  2.2]{Hetyei-EKP}. 
\begin{theorem}
\label{T_b2}
For $n\geq 1$, the number $\kappa_n$ of kernel positions of rank $n$ in the
original Bernoulli game is given by
$$
\kappa_n=(-1)^{n-1} (n+1)! b_n,
$$
where $b_n$ is the $n$-th Bernoulli number of the second kind. 
\end{theorem}
Here the Bernoulli number of the second kind $b_n$ is  obtained by
substituting zero into the Bernoulli polynomial of the second kind
$b_n(x)$, given by the generating function
\begin{equation}
\label{E_b2}
\sum_{n=0}^{\infty} \frac{b_n(x)}{n!} t^n=\frac{t(1+t)^x}{\ln(1+t)},
\end{equation}
see Roman~\cite[p.\ 116]{Roman}. Note that~\cite[p. 114]{Roman}
Jordan's~\cite[p.\ 279]{Jordan} earlier definition of the Bernoulli
polynomial of the second kind $\phi_n(x)$ is obtained by dividing $b_n(x)$ 
by $n!$. 

The proof of Theorem~\ref{T_b2} depends on a few simple observations
which were generalized in~\cite{Hetyei-EKP} to a class of {\em Bernoulli
  type games on posets} (see \cite[Definition 3.1]{Hetyei-EKP}). 
The set of positions $P$ in these games is a partially ordered 
set with a unique minimum element $\0$ and a rank 
function $\rho: P\rightarrow {\mathbb N}$ such that for each $n\geq 0$
the set $P_n$ of positions of rank $n$ have finitely many elements. 
The valid moves satisfy the following criteria:
\begin{itemize}
\item[(i)] Each valid move is from a position of higher rank to a
  position of lower rank. The set of positions reachable from a single
  position is a chain. 
\item[(ii)] If $y_1$ and $y_2$ are both reachable from $x$ in a single
  move and $y_1<y_2$ then $y_1$ is reachable from $y_2$ in a single move.
\item[(iii)] For all $m<n$ there is a number $\gamma_{m,n}$ such that 
each $y$ of rank $m$ may be reached from exactly $\gamma_{m,n}$
elements of rank $n$ in a single move. 
\end{itemize}
For such games, it was shown in~\cite[Proposition
  3.3]{Hetyei-EKP}, the numbers $\kappa_n$
of kernel positions of rank $n$ satisfy the recursion formula 
\begin{equation}
\label{E_gkrec}
|P_n|=\kappa_n+\sum_{m=0}^{n-1} \kappa_{m}\cdot \gamma_{m,n}.
\end{equation} 

\section{Winning a strongly Bernoulli type truncation game}
\label{s_Btg}

Let $\Lambda$ be an alphabet and let us denote by $\Lambda^*$
the free monoid generated by $\Lambda$, i.e., set 
$$\Lambda^*:=\{v_1\cdots v_n \::\: n\geq 0, \forall i (v_i\in
\Lambda)\}.$$ 
Note that $\Lambda^*$ contains the empty word $\varepsilon$. 
\begin{definition}
Given a subset $M\subseteq \Lambda^*\setminus \{\varepsilon\}$,
we define the {\em truncation game induced by $M$} as the game whose
positions are the elements of $\Lambda^*$, and whose valid moves consist of 
all truncations $v_1\cdots v_n\rightarrow v_1\cdots v_i$ such that 
$v_{i+1}\cdots v_n\in M$. 
\end{definition}
Note that $\varepsilon\not\in M$ guarantees that the truncation game
induced by $M$ is progressively finite, we may define the {\em rank} of
each position as the length of each word. This rank decreases after each
valid move. 
\begin{definition}
Given $M\subset \Lambda^*\setminus \{\varepsilon\}$, and  $P\subseteq
  \Lambda^*$, we say that $P$ is {\em $M$-closed} if for all
  $v_1\cdots v_n\in\Lambda^*\setminus\{\varepsilon\}$, $v_1\cdots v_n\in
  P$ and $v_{i+1}\cdots v_n\in M$ imply $v_1\cdots v_i\in P$. For an
  $M$-closed $P$, the {\em restriction of the truncation
 game induced by $M$ to $P$} is the game whose positions are the
elements of $P$ and whose valid moves consist of 
all truncations $v_1\cdots v_n\rightarrow v_1\cdots v_i$ such that 
$v_{i+1}\cdots v_n\in M$ and $v_1\cdots v_n\in P$.  
We denote this game by $(P,M)$, and call it the {\em truncation game
  induced by $M$ on $P$}.
\end{definition}
Clearly the definition of being $M$-closed is equivalent to saying that
the set $P$ is closed under making valid moves.
\begin{definition}
\label{D_Btm}
We say that $M\subset \Lambda^*\setminus\{\varepsilon\}$ {\em induces a
  Bernoulli type truncation game} if for all pairs of words $\underline{u},
\underline{v}\in \Lambda^*\setminus\{\varepsilon\}$,
  $\underline{u}\underline{v}\in M$ and $\underline{v}\in M$ imply
$\underline{u}\in M$.   
If $M$ is also closed under taking nonempty initial segments, i.e., 
$v_1\cdots v_n\in M$ implies $v_1\cdots v_m\in M$ for all
  $m\in\{1,\ldots,n\}$ then we say that $M$ induces a {\em strongly
  Bernoulli type truncation game}. If $M$ induces a
  (strongly) Bernoulli type truncation game, we call also $(P,M)$ a
  (strongly) Bernoulli type truncation game for each $M$-closed
  $P\subseteq \Lambda^*$. 
\end{definition}
Every strongly Bernoulli type truncation game is also a Bernoulli type
truncation game. The converse is not true: consider for example the set
$M$ of all words of positive even length. It is easy to see that the
truncation game induced by $M$ is Bernoulli type, but it is not
strongly Bernoulli type since $M$ is not closed under taking initial segments
of odd length. 

\begin{remark}
The definition of a Bernoulli type truncation game is {\em almost} a
special case of the Bernoulli type games on posets defined in~\cite[Definition
  3.1]{Hetyei-EKP}. Each $M$-closed $P\subseteq \Lambda^*$ is partially
ordered by the relation $v_1\cdots v_m<v_1\cdots v_n$ for all $m<n$,
the unique minimum element of this poset is $\varepsilon$, and
the length function is a rank function for this partial order.
For this poset and rank function, the set of valid
moves satisfies conditions (i) and (ii) listed in Subsection~\ref{s_b2}.
Only the ``uniformity'' condition (iii) and the finiteness of $|P_n|$ do
not need to be satisfied. These conditions were used
in~\cite{Hetyei-EKP} to prove equation (\ref{E_gkrec}) and count the
kernel positions of rank $n$ ``externally''. 
In this section we will show that the kernel positions 
of a strongly Bernoulli type truncation game on words may be described 
``internally'' in a manner that will allow their enumeration when each
$|P_n|$ is finite. The question whether the results presented in this
section may be generalized to all Bernoulli type truncation games remains
open. All examples of Bernoulli games played on words
in~\cite{Hetyei-EKP} are isomorphic to strongly Bernoulli type
truncation games, we will prove this for most of them in this paper,
the remaining examples are left to the reader. Together with the results
in~\cite{Hetyei-EKP}, we thus obtain two independent ways to count the
same kernel positions in these games. Comparing the results in~\cite{Hetyei-EKP}
with the results in the present paper yields explicit formulas 
for the coefficients in the Taylor expansion of certain functions. 
\end{remark}

In the rest of the section we set $P=\Lambda^*$ 
and just find the winning strategy for the truncation game
induced by $M$.  Only the formulas counting the kernel positions will change
when we change the set $P$ in the subsequent sections, the
decomposition of the kernel positions will not. 
First we define some {\em elementary kernel
  positions} in which the second player may win after at most one move
by the first player.  
\begin{definition}
\label{D_ekp}
The word $v_1\cdots v_n \in \Lambda^*\setminus\{\varepsilon\}$ is {\em
  an elementary kernel position} if it satisfies $v_1\cdots v_n\not \in
  M$, but for all $m<n$ we have $v_1\cdots v_m\in M$.
\end{definition}
In particular, for $n=1$, $v_1$ is an elementary kernel position
if and only if $v_1\not\in M$. Our terminology is justified by the following
two lemmas. 
\begin{remark}
\label{R_ekp1}
A position $v_1$ is a winning position, if and only if it is an
elementary kernel position. Otherwise it is not a kernel position at
all. 
\end{remark}
\begin{lemma}
\label{L_ekp2}
For $n>1$, starting from an elementary kernel position $v_1\cdots v_n$,
the first player is either unable to move, or is able to move only to a
position where the second player may win in a single move.
\end{lemma}
\begin{proof}
There is nothing to prove if the first player is unable to
move. Otherwise, by $v_1\cdots v_n\not\in M$, the first player is unable
to move to the empty word. Thus, after his or her move, we
arrive in a $v_1\cdots v_m$ where $1\leq m\leq n-1$. 
Thus $v_1\cdots v_m\in M$ holds, the second player may now move to
the empty word right away. 
\end{proof}
Next we show that the set of kernel positions in a strongly Bernoulli type
truncation game on $\Lambda^*$ is closed under the {\em concatenation}
operation.  
\begin{proposition}
\label{P_conc}
Let $\underline{u}:=u_1\cdots u_m$ be a kernel position of length $m\geq 1$ in a strongly Bernoulli truncation game induced by $M$. Then an arbitrary position  
$\underline{v}:=v_1\cdots v_n$ of length $n\geq 1$ is a kernel position
if and only if the concatenation $\underline{u}\underline{v}$ is also a
kernel position.  
\end{proposition}
\begin{proof}
Assume first that $\underline{u}\underline{v}$ is a kernel
position. We instruct the second player to play the winning strategy
that exists for $\underline{v}$ as long as the
length of the word truncated from $\underline{u}\underline{v}$
 at the beginning of his or her move is greater than $m$. 
For pairs of words longer than $m$, the validity of a move is
determined without regard to the letters in the first $m$ positions. By
playing the winning strategy for $\underline{v}$ as
long as possible, the second player is able to force the first player
into a position where the first player is either unable to move, or will
be the first to move to a word of length less than $m$, say
$u_1\cdots u_k$. The validity of this move implies $u_{k+1}\cdots u_m
v_1\cdots v_i\in M$ for some $i\geq 0$. By the strong Bernoulli
property we obtain $u_{k+1}\cdots u_m\in M$ and moving from
$u_1\cdots u_m$ to $u_1\cdots u_k$ is also a valid
move. We may thus pretend that the first player just made the first move
from $u_1\cdots u_m$ and the second player may win by
following the winning strategy that exists for $\underline{u}$.

For the converse, assume that $\underline{v}$  is
not a kernel position. In this case we may instruct the first player to 
play the strategy associated to $\underline{v}$ as long as possible, forcing 
the second player into a position where he or she is either unable to
move, or ends up making a move equivalent to a first move starting from 
$\underline{u}$. Now the original first player becomes
the second player in this subsequent game, and is able to
win. Therefore, in this case the concatenation 
$\underline{u}\underline{v}$ is not a kernel position either.
\end{proof} 
Using all results in this section we obtain the following structure
theorem. 
\begin{theorem}
\label{T_sBt}
A word $\underline{v}\in \Lambda^*\setminus\{\varepsilon\}$ is a kernel
  position in a strongly Bernoulli type truncation game, if and only if it may
be obtained by the concatenation of one or several elementary kernel
positions. Such a decomposition, if it exists, is unique.  
\end{theorem}
\begin{proof}
The elementary kernel positions are kernel positions by
Remark~\ref{R_ekp1} and Lemma~\ref{L_ekp2}. 
Repeated use of Proposition~\ref{P_conc}
yields that a pair of words obtained by concatenating several elementary
kernel positions is also a kernel position. 

For the converse assume that $\underline{v}:=v_1\cdots v_n$ is a kernel
position. We prove by induction on $n$ that this
position is either an elementary kernel position or may be obtained by
concatenating several elementary kernel positions. 
Let $m$ be the least index for which $v_1\cdots v_m\not\in M$ holds, such an
$m$ exists, otherwise the first player is able to move to $\varepsilon$
and win in the first move. It follows from the definition that
the $v_1\cdots v_m$ is an elementary kernel position. If $m=n$ then we
are done, otherwise applying Proposition~\ref{P_conc} to 
$v_1\cdots v_n=(v_1\cdots v_m) \cdot (v_{m+1}\cdots v_n)$
yields that $v_{m+1}\cdots v_n$ must be a kernel
position. We may apply the induction hypothesis to $v_{m+1}\cdots v_n$. 

The uniqueness of the decomposition may also be shown by induction on
$n$. Assume that $v_1\cdots v_n$ is a kernel position
and thus arises as a concatenation of one or several elementary kernel
positions. Let $v_1\cdots v_m$ be the leftmost factor in
this concatenation. By Definition~\ref{D_ekp}, $m$ is the least index
such $v_1\cdots  v_m\not\in M$ is satisfied. This determines the
leftmost factor uniquely. Now we may apply our induction hypothesis to
$v_{m+1}\cdots v_n$.  
\end{proof}

\section{The original Bernoulli game}
\label{s_ob2}

When we want to apply Theorem~\ref{T_sBt} to the original
Bernoulli game, we encounter two minor obstacles. The first obstacle is
that the rule defining a valid move from $(u_1\cdots u_n, v_1\cdots v_n)$ makes
an exception for the letters $u_1=v_1=1$, and does not allow their
removal. The second obstacle is that the game is defined on pairs of
words. Both problems may be easily remedied by changing the alphabet to
$\Lambda={\mathbb P}\times {\mathbb P}\times {\mathbb P}={\mathbb P}^3$ where
${\mathbb P}$ is the set of positive integers. 
\begin{lemma}
\label{L_oBiso}
The original Bernoulli game is isomorphic to the
strongly Bernoulli type truncation game induced by 
$$
M=\{(p_1,u_1,v_1)\cdots (p_n,u_n,v_n)\::\: p_1\neq 1, u_1\leq v_1,
\ldots, v_n\}, 
$$
on the set of positions 
$$
P=\{(1,u_1,v_1)\cdots (n,u_n,v_n)\::\: 1\leq u_i, v_i\leq i\}\subset
({\mathbb P}^3)^*. 
$$
The isomorphism is given by sending each pair of words 
$(u_1\cdots u_n,v_1\cdots v_n)\in ({\mathbb P}^2)^*$
into the word  $(1,u_1,v_1)(2,u_2,v_2)\cdots (n,u_n,v_n)\in ({\mathbb P}^3)^*$. 
\end{lemma}
Theorem~\ref{T_sBt} provides a new way of counting the kernel positions
of rank $n$ in the game $(P,M)$ defined in Lemma~\ref{L_oBiso}. 
Each kernel position $(1,u_1,v_1)\cdots (n,u_n,v_n)$ may be
uniquely written as a concatenation of elementary kernel positions. 
Note that these elementary kernel positions do not need to belong to the
set of valid positions $P$. However, we are able to
independently describe and count all elementary kernel positions 
that may appear in a concatenation factorization of a valid kernel position
$(1,u_1,v_1)\cdots (n,u_n,v_n)$ and contribute the segment $(i,u_i,v_i)\cdots
(j,u_j,v_j)$ to it.  
We call such a pair an {\em elementary kernel factor of type $(i,j)$}
and denote the number of such factors by $\kappa(i,j)$. Note that for
$i=1$ we must have $j=1$ and $(1,1,1)$ is the only elementary kernel
factor of type $(1,1)$. Thus we have $\kappa(1,1)=1$. 
\begin{lemma}
\label{L_ekf}
For $2\leq i\leq j$, a word $(i,u_i,v_i)\cdots (j,u_j,v_j)\in ({\mathbb
  P}^3)^*$ is an elementary kernel factor of type $(i,j)$ if and only if
it satisfies the following criteria: 
\begin{itemize}
\item[(i)] for each $k\in \{i,i+1,\ldots,j\}$ we have $1\leq u_k, v_k\leq k$;
\item[(ii)] we have $u_i>v_j$;
\item[(iii)] for all $k\in \{i,i+1,\ldots,j-1\}$ we have $u_i\leq v_k$.
\end{itemize}
\end{lemma}
In fact, condition (i) states the requirement for a valid position for
the letters at the positions $i,\ldots, j$, whereas conditions (ii) and
(iii) reiterate the appropriately shifted variant of the definition of an
elementary kernel position. A word $(1,u_1, v_1)\cdots
(n,u_n,v_n)$ that arises by concatenating $(1,u_1,v_1)\cdots (i_1,u_{i_1},
v_{i_1})$, $(i_1+1,u_{i_1+1},v_{i_1+1})\cdots (i_2,u_{i_2},v_{i_2})$,
and so on, $(i_{k+1,}u_{i_k+1},v_{i_k+1})\cdots (n,u_n,v_n)$ belongs to $P$
if and only if each factor $(i_s+1,u_{i_s+1},v_{i_s+1})\cdots
(i_{s+1},u_{i_{s+1}},v_{i_{s+1}})$ 
(where $0\leq s\leq k$, $i_0=0$ and $i_{k+1}=n$) satisfies conditions
(i) and (ii) in Lemma~\ref{L_ekf} with $i=i_s+1$ and $j=i_{s+1}$. We
obtain the unique factorization as a concatenation of elementary kernel
positions if and only if each factor $(u_{i_s+1},v_{i_s+1})\cdots
(u_{i_{s+1}},v_{i_{s+1}})$ also satisfies condition 
(iii) in Lemma~\ref{L_ekf} with $i=i_s+1$ and $j=i_{s+1}$. Using the
description given in Lemma~\ref{L_ekf} it is easy to calculate the
numbers $\kappa(i,j)$.
\begin{lemma}
\label{L_ekfc}
For $2\leq i\leq j$, the number of elementary kernel factors of type $(i,j)$ is
$$
\kappa(i,j)=(j-i)!^2\binom{j}{i}\binom{j}{i-2}.
$$
\end{lemma}
\begin{proof}
There is no other restriction on $v_{i+1},\ldots,v_{j}$ than the
inequality given in condition (i) of Lemma~\ref{L_ekf}. These numbers
may be chosen in $(i+1)(i+2)\cdots j=j!/i!$ ways. Let us denote
the value of $u_i$ by $u$, this must satisfy $1\leq u\leq i$. However,
$v_j<u_i$ may only be satisfied if $u$ is at least $2$. In that case 
$v_j$ may be selected in $(u-1)$ ways, and each $v_k$ (where $i\leq
k\leq j-1$ may be selected in $(k+1-u)$ ways (since $u_i\leq v_k\leq
k$). Thus the values of $v_i,\ldots v_j$ may be selected in 
$(u-1)\cdot (i+1-u)(i+2-u)\cdots (j-u)=(u-1)\cdot (j-u)!/(i-u)!$
ways. We obtain the formula
$$
\kappa(i,j)=\sum_{u=2}^{i} (u-1)\cdot
\frac{j!(j-u)!}{i!(i-u)!}
=
(j-i)!^2\binom{j}{i}
\sum_{u=2}^{i} \binom{u-1}{u-2}\cdot \binom{j-u}{i-u}.
$$     
Replacing the binomial coefficients with symbols
$$
\left(\binom{n}{k}\right):=\binom{n+k-1}{k},
$$
counting the $k$-element multisets on an
$n$-element set, we may rewrite the last sum 
as 
$$
\sum_{u=2}^{i} \left(\binom{2}{u-2}\right)\cdot
\left(\binom{j-i+1}{i-u}\right)=
\left(
\binom{j-i+3}{i-2}
\right).
$$
Thus we obtain 
$$
\kappa(i,j)=(j-i)!^2\binom{j}{i}
\left(
\binom{j-i+3}{i-2}
\right),
$$
which is obviously equivalent to the stated equation. 
\end{proof} 
Once we have selected the length of the elementary kernel factors in the
unique decomposition of a kernel position, we may select each kernel
factor of a given type independently. Thus we obtain the following
result.
\begin{theorem}
\label{T_b2k}
For $n\geq 1$, the number $\kappa_n$ of kernel positions of rank $n$ in
the original Bernoulli game is given by 
$$
\kappa_n=\sum_{k=0}^{n-2}\sum_{1=i_0<i_1<\cdots <i_{k+1}= n}
\prod_{j=0}^k (i_{j+1}-i_j-1)!^2
\binom{i_{j+1}}{i_j+1}\binom{i_{j+1}}{i_j-1}.  
$$ 
\end{theorem}
\begin{proof}
Consider the isomorphic game $(P,M)$ given in Lemma~\ref{L_oBiso}.
Assuming that the elementary kernel factors cover the positions $1$
through $1$, $2=i_0+1$ through $i_1$, $i_1+1$ through $i_2$, and so on, $i_k+1$
through $i_{k+1}=n$, we obtain the formula
$$
\kappa_n=\kappa(1,1)\sum_{k=0}^{n-1}\sum_{1=i_0<i_1<\cdots <i_{k+1}= n}
\prod_{j=0}^k \kappa(i_j+1,i_{j+1}), 
$$ 
from which the statement follows by $\kappa(1,1)=1$ and Lemma~\ref{L_ekfc}.
\end{proof}
Comparing Theorem~\ref{T_b2k} with Theorem~\ref{T_b2} we obtain the
following formula for the Bernoulli numbers of the second kind. 
\begin{corollary}
\label{C_b2} 
For $n\geq 2$ the Bernoulli numbers of the second kind are
given by 
\begin{equation}
\label{E_b2e}
b_n=(-1)^{n-1} \frac{1}{(n+1)!} 
\sum_{k=0}^{n-2}\sum_{1=i_0<i_1<\cdots <i_{k+1}= n}
\prod_{j=0}^k (i_{j+1}-i_j-1)!^2
\binom{i_{j+1}}{i_j+1}\binom{i_{j+1}}{i_j-1}.  
\end{equation}
\end{corollary}
\begin{example}
For $n=4$, Equation (\ref{E_b2e}) yields
\begin{align*}
b_4=\frac{-1}{5!}&\left((3-1)!^2\binom{4}{2}\binom{4}{0}
+(1-1)!^2\binom{2}{2}\binom{2}{0}(3-2)!^2\binom{4}{3}\binom{4}{1}\right.\\
&+(2-1)!^2\binom{3}{2}\binom{3}{0}(3-3)!^2\binom{4}{4}\binom{4}{2}\\
&\left.+(1-1)!^2\binom{2}{2}\binom{2}{0}
(2-2)!^2\binom{3}{3}\binom{3}{1}
(3-3)!^2\binom{4}{4}\binom{4}{2}\right)=-\frac{19}{30}.
\end{align*}
Thus $b_4/4!=-19/720$, which agrees with the number tabulated by
Jordan~\cite[p.\ 266]{Jordan}.
\end{example}
As $n$ increases, the number of terms in (\ref{E_b2e}) increases
exponentially. However, we are unaware of any other explicit formula
expressing the Bernoulli numbers of the second kind as a sum of terms of
the same sign. 

Lemma~\ref{L_ekfc} may also be used to obtain a recursion formula for
the number of kernel positions of rank $n$ in the original Bernoulli
game.
\begin{proposition}
\label{P_b2krec}
For $n\geq 2$, the number $\kappa_n$ of kernel positions of rank $n$ in
the original Bernoulli game satisfies the recursion formula 
$$
\kappa_n=\sum_{i=1}^{n-1} \kappa_i (n-i-1)!^2 \binom{n}{i+1}\binom{n}{i-1}.
$$
\end{proposition}
\begin{proof}
Consider again the isomorphic game $(P,M)$ given in Lemma~\ref{L_oBiso}.
Assume the last elementary kernel factor 
is $(i+1,u_{i+1}, v_{i+1})\cdots (n,u_n,v_n)$ where $i\geq 1$. Removing it we obtain a kernel
position of rank $i$. Conversely, concatenating an elementary  kernel
factor $(i+1,u_{i+1}, v_{i+1})\cdots (n,u_n,v_n)$ to a kernel position of
rank $i$ yields a kernel position of rank $n$. Thus we have 
\begin{equation}
\kappa_n=\sum_{i=0}^{n-1} \kappa_i \cdot \kappa(i+1,n),
\end{equation}
and the statement follows by Lemma~\ref{L_ekfc}.
\end{proof}
Comparing Proposition~\ref{P_b2krec} with Theorem~\ref{T_b2} we obtain the
following recursion formula for absolute values of the Bernoulli
numbers of the second kind. 
\begin{equation}
\label{E_b2rec}
|b_n|=\frac{1}{n+1}\sum_{i=1}^{n-1} |b_i| (n-i-1)!
 \binom{n}{i-1}\quad\mbox{holds for $n\geq 2$}.
\end{equation}
Equivalently, Jordan's~\cite{Jordan} Bernoulli numbers of the second
kind $b_n/n!$ satisfy 
\begin{equation}
\label{E_b2Jrec}
\left|\frac{b_n}{n!}\right|=\sum_{i=1}^{n-1} \left|\frac{b_i}{i!}\right| 
\frac{i}{(n+1)(n-i+1)(n-i)} \quad\mbox{for $n\geq 2$}.
\end{equation}
\begin{remark}
Since the sign of $b_n$ for $n\geq 1$ is $(-1)^{n-1}$, and
  substituting $x=0$ in (\ref{E_b2}) gives 
$$
\sum_{n\geq 0} \frac{b_n}{n!} t^n=\frac{t}{\ln(1-t)},
$$
it is easy to verify that (\ref{E_b2Jrec}) could also be derived from
the following equation, satisfied by the generating function of the
numbers $b_n$:  
$$
\frac{d}{dt}\left(t\cdot \frac{t}{\ln(1-t)}\right)+1-t
=
\frac{d}{dt}\left(\frac{t}{\ln(1-t)}\right)\cdot ((1-t)\ln(1-t)+t).
$$
However, it seems hard to guess that this equation will yield a nice
recursion formula.
\end{remark}

\section{Decomposing the indecomposable permutations}
\label{s_MR}

\begin{definition}
The {\em instant Bernoulli game} is the restriction of the original
Bernoulli game to the set of positions $\{ (12\cdots n,v_1\cdots
v_n)\::\: n\geq 1\}$.
\end{definition}
\begin{lemma}
\label{L_MRsimple}
Equivalently, we may define the set of positions of the
instant Bernoulli game as the set of words $v_1\cdots v_n$
satisfying $n\geq 1$ and $1\leq v_i\leq i$ for all $i$. A valid move
consists of replacing $v_1\cdots v_n$ with $v_1\cdots v_m$ for some
$m\geq 1$ such that $m+1\leq v_{m+1}, v_{m+2},\ldots, v_n$ holds.
\end{lemma}
Lemma~\ref{L_MRsimple} offers the simplest possible way to visualize the
instant Bernoulli game, even if this is not a form in which
the applicability of Theorem~\ref{T_sBt} could be directly seen. For
that purpose we need to note that the isomorphism of games stated
in Lemma~\ref{L_oBiso} may be restricted to the set of positions of the
instant Bernoulli game, and we obtain the following representation.
\begin{lemma}
\label{L_iBiso}
The instant Bernoulli game is isomorphic to the
strongly Bernoulli type truncation game induced by 
$$
M=\{(p_1,u_1,v_1)\cdots (p_n,u_n,v_n)\::\: p_1\neq 1, u_1\leq v_1,
\ldots, v_n\}, 
$$
on the set of positions 
$$
P=\{(1,1,v_1)\cdots (n,n,v_n)\::\: 1\leq u_i, v_i\leq i\}\subset
({\mathbb P}^3)^*. 
$$
\end{lemma}
Unless otherwise noted, we will use the simplified representation
stated in Lemma~\ref{L_MRsimple}. 
The kernel
positions of the instant Bernoulli game are identifiable with the
primitive elements of the Malvenuto-Reutenauer Hopf algebra, as it was
mentioned in the concluding remarks of~\cite{Hetyei-EKP}.
We call this game the instant Bernoulli game because this is a game
in which one of the players wins instantly: either there is no valid
move and the second player wins instantly, or the first player may
select the least $m\geq 1$ satisfying $m+1\leq v_{m+1}, v_{m+2},\ldots,
v_n$ and move to $v_1\cdots v_m$, thus winning instantly. The kernel
positions are identical to the winning positions in 
this game. The recursion formula (\ref{E_gkrec}) may be rewritten as
$$
n!=\kappa_n+\sum_{m=1}^{n-1} \kappa_m (n-m)!,
$$ 
(we start the summation with $\kappa_1$ since the first letter cannot
be removed), and the generating function of the numbers $\kappa_n$ is
easily seen to be 
\begin{equation}
\label{E_MRgf}
\sum_{n=1}^{\infty} \kappa_n t^n=1-\frac{1}{\sum_{n=0}^{\infty}n!t^n}.
\end{equation}
The numbers $\{\kappa_n\}_{n\geq 0}$ are listed 
as sequence A003319 in the On-Line Encyclopedia of Integer
Sequences~\cite{OEIS}, and count the number of {\em connected} or {\em
  indecomposable} permutations of $\{1,2,\ldots,n\}$. A permutation
$\pi\in S_n$ is {\em connected} if there is no $m<n$ such that $\pi$
takes the set $\{1,\ldots,m\}$ into itself. The kernel positions of the
instant Bernoulli game are directly identifiable with the connected
permutations in more than one ways. One way is mentioned at the end
of~\cite{Hetyei-EKP}, we may formalize that bijection using two variants of the 
well-known {\em inversion tables} (see, for example ~\cite[Section
  5.1.1]{Knuth} or \cite[Section 1.3]{Stanley-EC1}).  
\begin{definition}
Given a permutation $\pi\in S_n$ we define its {\em letter-based
  non-inversion table} as the word $v_1\cdots v_n$ where 
$v_j=1+|\{i<j\::\: \pi^{-1}(i)<\pi^{-1}(j)\}|$.
\end{definition}   
For example, for $\pi=693714825$ the letter-based non-inversion table
is $121351362$. This is obtained by adding $1$ to all entries in the
usual definition of an inversion table~\cite[Section 1.3]{Stanley-EC1}
of the permutation $\widetilde{\pi}=417396285$, defined by
$\widetilde{\pi}(i)=n+1-\pi(i)$ and taking the reverse of the resulting
word. In particular, for $\widetilde{\pi}=417396285$ we find the
inversion table $(1,5,2,0,4,2,0,1,0)$ in~\cite[Section
  1.3]{Stanley-EC1}. Our term {\em letter-based} refers to the fact that
here we associate the letter $j$ to $v_j$ and not the place $j$.    

A variant of the notion of letter-based non-inversion table is the
place-based non-inversion table. 

\begin{definition}
Given a permutation $\pi\in S_n$ we define its {\em place-based
  non-inversion table (PNT)} as the word $v_1\cdots v_n$ where 
$v_j=1+|\{i<j\::\: \pi(i)<\pi(j)\}|$.
\end{definition}   
Obviously the PNT of a permutation $\pi$ equals the letter-based
non-inversion table of $\pi^{-1}$. For example, for $\pi=583691472$ the PNT is 
$121351362$. We have $v_7=3=1+2$ because $\pi(7)=4$ is preceded by two
letters $\pi(i)$ such that $(\pi(i),\pi(7))$ is not an inversion. Any
PNT $v_1\cdots v_n$ is a word satisfying $1\leq i\leq v_i$.

\begin{lemma}
\label{L_MRc}
A position $v_1\cdots v_n$ in the instant Bernoulli game is a kernel
position if and only if it is the place-based (letter-based)
non-inversion table of a connected permutation.
\end{lemma}
\begin{proof}
We prove the place-based variant of the lemma, the letter-based version
follows immediately since the set of connected permutations is closed
under taking inverses. It is easy to verify that the place-based
non-inversion table $v_1\cdots v_n$ of a permutation $\pi$ satisfies 
$m+1\leq v_{m+1},\ldots, v_n$ if and only if $\pi$ takes the set
$\{1,\ldots,m\}$ into itself. Thus the first player has no valid move if
and only if $\pi$ is connected.
\end{proof}
The study of connected permutations goes back to the work of
Comtet~\cite{Comtet-n!,Comtet-AC}, for a reasonably complete
list of references we refer to the entry A003319 in the On-Line
Encyclopedia of Integer Sequences~\cite{OEIS}. It was shown by Poirier
and Reutenauer~\cite{Poirier-Reutenauer} that the connected permutations
form a free algebra basis of the Malvenuto-Reutenauer Hopf-algebra,
introduced by Malvenuto and Reutenauer~\cite{Malvenuto-Reutenauer}.
The same statement appears in dual form in the work of Aguiar and
Sottile~\cite{Aguiar-Sottile}. 

Although the instant Bernoulli game is very simple,
Theorem~\ref{T_sBt} offers a nontrivial analysis of its kernel
positions, allowing to identify a unique structure on each connected
permutation. We begin with stating the following analogue of
Theorem~\ref{T_b2k}.

\begin{theorem}
\label{T_MRk}
The number $\kappa_n$ of connected permutations of rank $n$ 
is given by 
$$
\kappa_n=\sum_{k=1}^{n-1}\sum_{1\leq i_1<i_2<\cdots <i_{k+1}=n}
\prod_{j=1}^k (i_{j+1}-i_j-1)!\cdot i_j.
$$
\end{theorem}  
\begin{proof}
By Lemma~\ref{L_MRc}, $\kappa_n$ is the number of kernel positions of
rank $n$ in the instant Bernoulli game. The fact that this number is
equal to the expression on the right hand side may be shown similarly to
the proof of Theorem~\ref{T_b2k}. Consider the equivalent representation of the
instant Bernoulli game given in Lemma~\ref{L_MRsimple}. Note that this
is obtained from the representation given in Lemma~\ref{L_iBiso} by
deleting the ``redundant coordinates'' $i,i$ from each letter
$(i,i,v_i)$. Given an
arbitrary kernel position $v_1\cdots v_n$, the first letter $v_1=1$
corresponds to an elementary kernel factor of type $(1,1)$ 
and we have $\kappa(1,1)=1$. 
For $2\leq i\leq  j$, by abuse of terminology, 
let us call $v_i\cdots v_j$ an elementary kernel factor of type
$(i,j)$ if it corresponds to an elementary kernel factor in the
equivalent representation in Lemma~\ref{L_iBiso}.
The elementary kernel factors of type $(i,j)$ are then exactly
those words $v_i\cdots v_j$ for which $i\leq v_i,\ldots, v_{j-1}$ and $v_j<i$
hold. Thus their number is 
\begin{equation}
\label{E_MRij}
\kappa(i,j)=(j-i)!\cdot (i-1),
\end{equation}
The statement now follows from the obvious formula 
$$
\kappa_n=\kappa(1,1)\cdot \sum_{k=1}^{n-1}
\sum_{1=i_0<i_1<\cdots<i_{k+1}=n}
\prod_{j=1}^k \kappa(i_j+1,i_{j+1}).
$$
\end{proof}
In analogy to Proposition~\ref{P_b2krec}, we may also use (\ref{E_MRij})
to obtain a recursion formula for the number of connected permutations.
We end up with a formula that was first discovered by King~\cite[Theorem
  4]{King}. 
\begin{proposition}[King]
\label{P_MRrec}
For $n\geq 2$, the number $\kappa_n$ of connected permutations of rank
$n$ satisfies the recursion formula
$$
\kappa_n=\sum_{i=1}^{n-1} \kappa_i (n-i-1)!i.
$$
\end{proposition}
The proof may be presented the same way as for
Proposition~\ref{P_b2krec}, by removing the last elementary kernel
factor of type $(i,n)$, using informal notion of an elementary kernel
factor as in the proof of Theorem~\ref{T_MRk}. King's proof is worded
differently, but may be shown to yield a bijectively equivalent decomposition. 
\begin{lemma}
The induction step presented in King's proof of Proposition~\ref{P_MRrec} is
equivalent to the removal of the last
elementary kernel factor in the place-based non-inversion table of 
$\widetilde{\sigma}(1)\widetilde{\sigma}(2)\cdots\widetilde{\sigma}(n)$. 
Here $\widetilde{\sigma}(i)=n+1-\sigma(n+1-i)$.
\end{lemma} 
\begin{proof}
Let $\sigma(1)\cdots \sigma(n)$ be the connected permutation considered
in King's proof, and let $v_1\cdots v_n$ be the PNT of
$\widetilde{\sigma}(1)\widetilde{\sigma}(2)\cdots 
\widetilde{\sigma}(n)$. King's proof first identifies $\sigma(1)=r$. 
This is equivalent to setting $v_n=n+1-r$. King
then defines $\pi(1)\cdots \pi(n-1)$ as the permutation obtained by
deleting $\sigma(1)$ and subtracting $1$ from all letters greater than
$r$. Introducing $\widetilde{\pi}(i)=n-\pi(n-i)$, the permutation
$\widetilde{\pi}(1)\cdots\widetilde{\pi}(n-1)$ is obtained from
$\widetilde{\sigma}(1)\widetilde{\sigma}(2)\cdots \widetilde{\sigma}(n)$ 
by deleting the last letter $n+1-r$ and by decreasing all letters
greater than $n+1-r$ by one. The PNT of 
$\widetilde{\pi}(1)\cdots\widetilde{\pi}(n-1)$ is
thus $v_1\cdots v_{n-1}$. King then defines $j$ as the largest $j$
such that $\pi(\{1,\ldots,j\})=\{1,\ldots,j\}$. This is equivalent to
finding the least $n-j$ such that
$\widetilde{\pi}(\{n-j,n-j+1,\ldots,n-1\})=\{n-j,n-j+1,\ldots,n-1\}$. 
Using the proof of Lemma~\ref{L_MRc}, this is easily seen to be
equivalent to finding the smallest $n-j$ such that $v_{n-j}=n-j$ and for
all $n-j\leq k\leq n-1$ we have $v_{k}\geq n-j$. King defines
$\beta(\pi)$ as the permutation obtained from $\pi$ by removing
$\pi(1)\cdots \pi(j)$ and then subtracting $j$ from each
element. Correspondingly, we may define
$\widetilde{\beta}(\widetilde{\pi})$ as the permutation obtained from 
$\widetilde{\pi}$ by removing $\widetilde{\pi}(n-j)\cdots
\widetilde{\pi}(n-1)$. The PNT of
$\widetilde{\beta}(\widetilde{\pi})$ is then $v_1\cdots v_{n-j-1}$,
representing a kernel position in the instant Bernoulli game. This is
the kernel position of the least rank that is reachable from $v_1\cdots
v_{n-1}$. In terms of elementary kernel factors, the removal of $v_{n}$
makes the first player able to remove the rest of the last elementary kernel
factor in a single valid move, we only need to show that the fist player
cannot move to a position $v_1\cdots v_k$ where $r\leq k\leq s$ for some
elementary kernel factor $v_r\cdots v_s$. Assume by way of contradiction
that such a move is possible. By definition of a valid move,
we then have $k\leq v_s$, implying  $r\leq v_s$, in contradiction with
the definition of the elementary kernel factor $v_r\cdots
v_s$. Therefore $v_1\cdots v_{n-j-1}$ is obtained from $v_1\cdots v_n$
by removing exactly the last elementary kernel factor.
\end{proof}

King~\cite{King} uses the removal of the last elementary kernel factor
to recursively define a {\em transposition Gray code} of all connected
permutations of a given rank. A transposition Gray code is a list of 
permutations such that subsequent elements differ by a transposition.
Using place-based non-inversion tables, not only the last elementary
kernel factor is easily identifiable, but the entire unique
decomposition into elementary kernel factors is transparent. This gives 
rise to a new way to systematically list all connected permutations.
The resulting list is not a transposition Gray code, but it is fairly 
easy to generate. 

To explain the construction, consider the connected permutation 
$\pi=251376948$. Its letter-based non-inversion table is $v_1\cdots
v_8=121355748$ whose decomposition into elementary kernel factors is 
$1\cdot 21\cdot 3\cdot 5574\cdot 8$. For $i<j$, each elementary kernel
factor of type $(i,j)$ begins with $i$, all entries in the factor are at
least $i$, except for the last letter which is less than $i$.  
For $i=1$, $1$ is a special elementary kernel factor, for $i>1$ a kernel
factor of type $(i,i)$ is a positive integer less than $i$.
\begin{definition}
Given a connected permutation $\pi$, we define its {\em elevation $E(\pi)$} as
the permutation whose PNT is obtained from 
the PNT of $\pi$ as follows: for 
each elementary kernel factor of type $(i,j)$, increase the last letter in
the factor to $j$. 
\end{definition}
For example, the PNT of the elevation of $251376948$ is 
$1\cdot 23\cdot 4\cdot 5578\cdot 9$, thus $E(\pi)$ is 
$123465789$. The PNT of $E(\pi)$ is written as a product of factors,
such that each factor $u_i\cdots u_j$ ends with $j$,
and all letters after $u_j$ are more than $j$. We may use this
observation to prove that each factor $u_i\cdots u_j$ ends with a $j$
that is a {\em strong fixed point} $j$.
\begin{definition}
A number $i\in\{1,\ldots,n\}$ is a strong fixed point of a permutation
$\sigma$ of $\{1,\ldots,n\}$ if $\sigma(i)=i$ and
$\sigma(\{1,\ldots,i\})=\{1,\ldots,i\}$. We denote the set of strong
fixed points of $\sigma$ by $\SF(\sigma)$.
\end{definition}
\begin{remark}
The definition of a strong fixed point may be found in Stanley's
book~\cite[Ch.\ 1, Exercise 32b]{Stanley-EC1}, where it is stated that
the number $g(n)$ of permutations of rank $n$ with no strong fixed points
has the generating function 
$$
\sum_{n\geq 0} g(n) t^n=\frac{\sum_{n\geq 0} n! t^n}{1+t \sum_{n\geq 0} n! t^n}.
$$ 
\end{remark}
\begin{lemma}
\label{L_ifp}
Let $v_1\cdots v_n$ be the PNT of a permutation $\sigma$. Then $j$ is an
 strong fixed point of $\sigma$ if and only if $v_j=j$ and for all $k>j$
 we have $v_k>j$.
\end{lemma}
In fact, the condition $\forall j (k>j\implies v_k>j)$ is easily seen to
be equivalent to $\sigma(\{1,\ldots,j\})=\{1,\ldots,j\}$. Assuming this
is satisfied, $j$ is a fixed point of $\sigma$ if and only if $v_j=j$.
As a consequence of Lemma~\ref{L_ifp} the last letters of the elementary 
kernel factors of the PNT of $\pi$ mark strong fixed points of
$E(\pi)$. The converse is not necessarily true: in our example $7$ is an
strong fixed point of $E(\pi)$; however, no elementary kernel factor of
the PNT of $\pi$ ends with $v_7$. On the other hand, $v_1$ is always a special
elementary kernel factor by itself and the last elementary
kernel factor must end at $v_n$, thus $1$ and $n$ must always be
strong fixed points of $E(\pi)$. The numbers $1$ and $n$ are also special
in the sense that $i\in\{1,n\}$ is an strong fixed point if and only if
it is a fixed point. 

\begin{theorem}
\label{T_epi}
Let $\sigma\in S_n$ be a permutation satisfying $\sigma(1)=1$ and
$\sigma(n)=n$ and let the strong fixed points of $\sigma$ be 
$1=i_0<i_1<\cdots<i_{k+1}=n$. Then there are
exactly $(i_1+1)\cdots (i_k+1)$ connected permutations $\pi$ whose elevation is
$\sigma$. 
\end{theorem}
\begin{proof}
Assume $E(\pi)=\sigma$ and the PNT of $\pi$ is the product of
elementary factors of type $(1,1)$, $(j_0+1,j_1)$, $(j_1+1,j_2)$, \ldots, 
$(j_l+1,j_{l+1})$, where $1=j_0<j_1<\cdots<j_{l+1}=n$. As we have seen
above, $\{j_1,\ldots,j_l\}$ must be a subset of $\{i_1,\ldots,i_k\}$.
This condition is also sufficient since we may decompose the PNT of
$\sigma$ as $u_1\cdot (u_{j_0+1}\cdots u_{j_1})\cdots (u_{j_l+1},u_{j_{l+1}})$,
and decrease the value of each $u_{j_t}=j_t$ (where $t=1,2,\ldots,l+1$)
independently to any number that is at most $j_{t-1}$. Note that
each $u_{j_t+1}=j_t+1$, and the required inequalities for all other $u_j$s are
automatically satisfied as a consequence of having selected the $j_t$s
from among the strong fixed points. Thus we obtain the PNT
of a connected permutation, whose kernel factors are of type 
$(1,1)$, $(j_0+1,j_1)$, $(j_1+1,j_2)$, \ldots, 
$(j_l+1,j_{l+1})$. Therefore the number of permutations $\pi$ satisfying 
$E(\pi)=\sigma$ is 
$$
\sum_{l=0}^k \sum_{\{j_1,\ldots,j_l\}\subseteq \{i_1,\ldots,i_k\}}
j_1\cdots j_l=(i_1+1)\cdots (i_k+1).
$$
\end{proof}
The proof of Theorem~\ref{T_epi} suggests a straightforward way to list
the PNTs of all connected permutations of rank $n$:
\begin{itemize}
\item[(1)] List all words $u_1\cdots u_n$ satisfying $u_1=1$, $u_n=n$
  and  $1\leq u_i\leq i$ for all $i$. These are the PNTs of all
  permutations of rank $n$, of which $1$ and $n$ are fixed points.
\item[(2)] For each $u_1\cdots u_n$, identify the places of strong
  fixed points by finding all $i$s such that $u_i=i$ and $u_k>i$ for all
  $k>i$.
\item[(3)] For each $u_1\cdots u_n$ select a subset
  $\{j_1,\ldots,j_l\}$ of the set of strong fixed points satisfying
$1<j_1<\cdots<j_l<n$ and decrease the values of each $u_{j_t}$ 
to any number in $\{1,\ldots,j_{t-1}\}$. Output these as the PNTs of
  connected permutations. 
\end{itemize}
Steps and $(1)$ and $(3)$ involve nothing more than listing words using
some lexicographic order, step $(2)$ may be performed after reading each
word once.  

As a consequence of Theorem~\ref{T_epi} we obtain the following formula
for the number of connected permutations of rank $n\geq 2$:
$$
\kappa_n=\sum_{\substack{\sigma\in S_n\\\sigma(1)=1, \sigma(n)=n}} \prod_{i\in
  \SF(\sigma)\setminus\{1,n\}} (i+1).
$$
After removing the redundant letters $\sigma(1)=1$ and $\sigma(n)=n$ and
decreasing all remaining letters by $1$, we obtain that
\begin{equation}
\label{E_IF}
\kappa_n=\sum_{\sigma\in S_{n-2}} \prod_{i\in
  \SF(\sigma)} (i+2)\quad\mbox{holds for $n\geq 2$}.
\end{equation}
Equation (\ref{E_IF}) offers a new combinatorial model for the numbers
counting the connected permutations of rank $n\geq 2$: it is the total weight
of all permutations of rank $n-2$, using a weighting which assigns the
most value to those permutations which have the most strong fixed points
and are thus in a sense the farthest from being connected. 

\section{The polynomial Bernoulli game of the second kind, indexed by
$x$}
\label{s_pb2}

This game is defined in~\cite{Hetyei-EKP} on triplets of words 
$(u_1\cdots u_n, v_1\cdots v_n, w_1\cdots w_n)$ for
  $n\geq 0$ such that $1\leq u_i\leq i$, $1\leq v_i\leq i+1$ and 
$1\leq w_i\leq x$ hold for $i\geq 1$, furthermore we require 
$w_i\leq w_{i+1}$ for all $i\leq n-1$. 
A valid move consists of replacing $(u_1\cdots
u_n,v_1\cdots v_n, w_1\cdots w_n)$ with $(u_1\cdots u_m,v_1\cdots v_m,
w_1\cdots w_m)$ for some $m\geq 0$ satisfying
$w_{m+1}=w_{m+2}=\cdots=w_n=x$ and $u_{m+1}< v_j$ for $j=m+1,\ldots, n$.
Theorem~\ref{T_sBt} is applicable to this game, because of the following
isomorphism. 
\begin{lemma}
\label{L_rpb2}
Let $\Lambda={\mathbb P}\times {\mathbb P}\times \{1,\ldots,x\}$ where $x\in
 {\mathbb P}$. The polynomial Bernoulli game, indexed by $x$ is
 isomorphic to the strongly Bernoulli type truncation game, induced by 
$$
M:=\{(u_1,v_1,x)\cdots (u_n,v_n,x)\::\: u_1< v_1, \ldots, v_n\}
$$
on the set of positions
$$
P:=\{(u_1,v_1,w_1)\cdots(u_n,v_n,w_n)\::\: 1\leq u_i\leq i, 1\leq
v_i\leq i+1, w_1\leq\cdots\leq w_n\}.
$$
This isomorphism is given by sending each triplet 
$(u_1\cdots u_n, v_1\cdots v_n, w_1\cdots w_n)\in {\mathbb P}^*\times
{\mathbb P}^* \times \{1,\ldots,x\}^*$ into
$(u_1,v_1,w_1)\cdots (u_n,v_n,w_n)\in ({\mathbb P}\times {\mathbb
  P}\times \{1,\ldots,x\})^*$. 
\end{lemma}
\begin{theorem}
\label{T_b2p}
The number $\kappa_n$ of kernel positions of rank $n$ in the polynomial
Bernoulli game of the second kind, indexed by $x$ is 
\begin{align*}
\kappa_n=&\sum_{m=0}^{n-1} \binom{x+m-2}{m} m!(m+1)! 
  \sum_{k=0}^{n-m-1}\sum_{m=i_0<i_1<\cdots <i_{k+1}=n}
  \prod_{j=0}^k (i_{j+1}-i_j-1)!^2
    \binom{i_{j+1}}{i_j+1}\binom{i_{j+1}+1}{i_j}\\
&+\binom{x+n-2}{n} n!(n+1)!.
\end{align*}
\end{theorem}
\begin{proof}
Consider the isomorphic game $(P,M)$, given in Lemma~\ref{L_rpb2}.
Since in a valid move all truncated letters $(u_j,v_j,w_j)$ satisfy 
$w_j=x$, we have to distinguish two types of elementary kernel factors:
those which contain a letter $(u_i,v_i,w_i)$ with $w_i<x$ and those
which do not. If the elementary kernel factor contains a $(u_i,v_i,w_i)$
with $w_i<x$, it must consist of the single letter $(u_i,v_i,w_i)$.  We
call such a factor an {\em elementary kernel factor of type
  $(i;w_i)$}. Clearly, their number is  
\begin{equation}
\label{E_iw}
\kappa(i;w_i)=i(i+1),
\end{equation}
since $u_i\in \{1,\ldots,i\}$ and $v_i\in \{1,\ldots,
i+1\}$ may be selected independently. The elementary kernel factors containing
only $x$ in their $w$-component of their letters are similar to the ones
considered in Lemma~\ref{L_ekf}. 
We call an elementary kernel factor of type $(i,j;x)$ an
elementary kernel factor $(u_i,v_i,x)\cdots (u_j, v_j, x)$.  
A calculation completely analogous to the one in Lemma~\ref{L_ekfc}
shows that their number is 
\begin{equation}
\label{E_ijx}
\kappa(i,j;x)=\sum_{u=1}^i
u\frac{(j-u)!j!}{(i-u)!i!}=(j-i)!^2\binom{j}{i}\binom{j+1}{i-1}. 
\end{equation}
Because of $w_1\leq \cdots \leq w_n$, the factors of type $(i;w_i)$
must precede the factors of type $(i,j;x)$. Thus we obtain 
\begin{align*}
\kappa_n=&
\sum_{m=0}^{n-1}
\sum_{1\leq w_1\leq \cdots \leq w_m\leq x-1} \prod_{i=1}^m \kappa(i;w_i)
  \sum_{k=0}^{n-m-1}\sum_{m=i_0<i_1<\cdots <i_{k+1}=n}
  \prod_{j=0}^k \kappa(i_j+1,i_{j+1};x)\\
&+\sum_{1\leq w_1\leq \cdots \leq w_n\leq x-1} \prod_{i=1}^n \kappa(i;w_i)
\end{align*}  
The statement now follows from (\ref{E_iw}), (\ref{E_ijx}), from
$\prod_{i=1}^m i(i+1)=m!(m+1)!$, 
and from the fact that the number of words  $w_1\cdots w_m$ satisfying 
$1\leq w_1\leq \cdots \leq w_m\leq x-1$ is 
$$
\left(\binom{x-1}{m}\right)=\binom{x+m-2}{m}.
$$
\end{proof}
We already know~\cite[Theorem 4.2]{Hetyei-EKP} that we also have
$$\kappa_n=(-1)^n(n+1)!b_n(-x)$$
for all positive integer $x$. Since two polynomial functions are equal
if their agree for infinitely many substitutions, we obtain a valid
expansion of the polynomial $(-1)^n(n+1)!b_n(-x)$. 
Substituting $-x$ into $x$ and
rearranging yields the expansion of $b_n(x)$ in the basis
$\{\binom{x+1}{n}\::\: n\geq 0\}$.
\begin{corollary}
\label{C_b2p}
Introducing $c_{n,n}=n!$ and
$$
c_{n,m}=\frac{(-1)^{n-m}m!(m+1)!}{(n+1)!} 
  \sum_{k=0}^{n-m-1}\sum_{m=i_0<i_1<\cdots <i_{k+1}=n}
  \prod_{j=0}^k (i_{j+1}-i_j-1)!^2
    \binom{i_{j+1}}{i_j+1}\binom{i_{j+1}+1}{i_j}
$$
for $0\leq m<n$, we have
$$
b_n(x)=\sum_{m=0}^{n} c_{n,m}\binom{x+1}{m}. 
$$
\end{corollary}
\begin{example}
For $n=2$, Corollary~\ref{C_b2p} gives
\begin{align*}
b_2(x)=& \frac{0!1!}{3!}\left(1!^2\binom{2}{1}\binom{3}{0}
  +0!^2\binom{1}{1}\binom{2}{0}0!^2\binom{2}{2}\binom{3}{1}\right)
-\frac{1!2!}{3!}\binom{x+1}{1}
  0!^2\binom{2}{2}\binom{3}{1}+\binom{x+1}{2}2!\\ 
=&\frac{5}{6}-(x+1)+(x+1)x=x^2-\frac{1}{6}.
\end{align*}
Thus $b_2(x)/2!=x^2/2-1/12$ which agrees with the formula given
in~\cite[\S 92]{Jordan}.
\end{example}
We may also obtain a new formula for the Bernoulli numbers of the second
kind by substituting $x=0$ into Corollary~\ref{C_b2p}. We obtain
$b_n=c_{n,0}+c_{n,1}$, i.e., 
\begin{equation}
\begin{aligned}
b_n=&\frac{(-1)^{n}}{(n+1)!} 
  \sum_{k=0}^{n-1}\sum_{0=i_0<i_1<\cdots <i_{k+1}=n}
  \prod_{j=0}^k (i_{j+1}-i_j-1)!^2
    \binom{i_{j+1}}{i_j+1}\binom{i_{j+1}+1}{i_j}\\
&+
\frac{(-1)^{n-1}\cdot 2}{(n+1)!} 
  \sum_{k=0}^{n-2}\sum_{1=i_0<i_1<\cdots <i_{k+1}=n}
  \prod_{j=0}^k (i_{j+1}-i_j-1)!^2
    \binom{i_{j+1}}{i_j+1}\binom{i_{j+1}+1}{i_j}
\end{aligned}
\end{equation}
for $n\geq 2$.

\section{The flat Bernoulli game}
\label{s_fB}

This game is defined in~\cite{Hetyei-EKP} on words $u_1\cdots
u_n$ for $n\geq 0$ such that each $u_i\in{\mathbb P}$ satisfies
$1\leq u_i\leq i$. A valid move consists of replacing $u_1\cdots
u_n$ with $u_1\cdots u_m$ if $m\geq 1$ and $u_{m+1}<u_j$ holds for all $j>m+1$.
In analogy to Lemma~\ref{L_oBiso}, we have the following result. 
\begin{lemma}
\label{L_fBiso}
The flat Bernoulli game is isomorphic to the
strongly Bernoulli type truncation game induced by 
$$
M=\{(p_1,u_1)\cdots (p_n,v_n)\::\: p_1\neq 1, u_1<u_2,\ldots, u_n\}, 
$$
on the set of positions 
$$
P=\{(1,u_1)\cdots (n,u_n)\::\: 1\leq u_i\leq i\}\subset
({\mathbb P}^2)^*. 
$$
The isomorphism is given by sending each word 
$u_1\cdots u_n\in {\mathbb P}^*$
into the word  $(1,u_1)(2,u_2)\cdots (n,u_n)\in ({\mathbb P}^2)^*$. 
\end{lemma}
\begin{theorem}
\label{T_fB}
For $n\geq 2$, the number $\kappa_n$ of kernel positions of rank $n$ in
the flat Bernoulli game is 
$$
\kappa_n=\sum_{k=0}^{\lfloor(n-3)/2\rfloor}
\sum_{\substack{1=i_0<i_1<\cdots<i_{k+1}=n\\ i_{j+1}-i_j\geq 2}}
\prod_{j=0}^k (i_{j+1}-i_j-2)!\binom{i_{j+1}}{i_j}.
$$
\end{theorem}
\begin{proof}
Consider isomorphic representation given in Lemma~\ref{L_fBiso}. 
Note first that, in any
kernel position $(1,u_1)\cdots (n,u_n)$, the letter $u_1=(1,1)$ is
an elementary kernel factor of type $(1,1)$ and we have $\kappa(1,1)=1$. 
For $2\leq i< j$, let $\kappa(i,j)$ be the number of elementary 
kernel factors $(i,u_i)\cdots (j,u_j)$ of type
$(i,j)$. A calculation completely analogous to the one in Lemma~\ref{L_ekfc}
shows
\begin{equation}
\label{E_ijf}
\kappa(i,j)=\sum_{u=1}^i u \frac{(j-1-u)!}{(j-1-i)!}=(j-1-i)!\binom{j}{i-1}.
\end{equation}
Note that for $i\geq 2$ there is no elementary kernel factor of type
$(i,i)$ since removing the last letter only is always a valid move,
provided at least one letter is left. The statement now follows from 
Equation (\ref{E_ijf}) and the obvious formula 
$$
\kappa_n=\kappa(1,1)\cdot \sum_{k=0}^{\lfloor(n-3)/2\rfloor}
\sum_{\substack{1=i_0<i_1<\cdots<i_{k+1}=n\\ i_{j+1}-i_j\geq 2}}
\prod_{j=0}^k \kappa(i_j+1,i_{j+1}).
$$
\end{proof}
Introducing $m_j:=i_{j}-i_{j-1}-1$ for $j\geq 1$ and shifting the index
$k$ by $1$, we may rewrite the equation in Theorem~\ref{T_fB} as 
\begin{equation}
\label{E_fB} 
\kappa_n=n\cdot \sum_{k=1}^{\lfloor(n-1)/2\rfloor}
\sum_{\substack{m_1+\cdots + m_{k}=n-1\\ m_1,\ldots, m_k\geq 2}}
\binom{n-1}{m_1, \ldots, m_k} (m_1-2)!\cdots (m_k-2)!.
\end{equation} 
A more direct proof of this equation follows from
Corollary~\ref{C_fBperm} below. 
\begin{example}
For $n=5$, (\ref{E_fB}) yields
$$
\kappa_5=5\left(\binom{4}{4} 2!+\binom{4}{2,2}0!0!\right)=40.
$$
Thus $\kappa_5/5!=1/3$ which agrees with the number given in~\cite[Table
  1]{Hetyei-EKP}. 
\end{example}
We already know~\cite[Proposition 7.3]{Hetyei-EKP} that the exponential
generating function of the numbers $\kappa_n$ is
\begin{equation}
\label{E_fBgen}
\sum_{n=1}^{\infty}\frac{\kappa_n}{n!} t^n
=\frac{t}{(1-t)(1-\ln(1-t))}.
\end{equation}
Just like in Section~\ref{s_MR}, we may use place-based non-inversion
tables to find a permutation enumeration model for
for the numbers $\kappa_n$.
\begin{lemma}
\label{L_PNT<}
Let $u_1\cdots u_n$ be the PNT of a permutation $\pi\in S_n$. 
Then, for all $i<j$, $\pi(i)<\pi(j)$ implies $u_i<u_j$.
The following partial converse is also true: $u_i<u_{i+1},\ldots, u_j$
implies $\pi(i)<\pi(i+1),\ldots, \pi(j)$.  
\end{lemma}
\begin{proof}
If $\pi(i)<\pi(j)$ then the set $\{k<i\::\: \pi(k)<\pi(i)\}$ is a proper
subset of $\{k<j\::\: \pi(k)<\pi(j)\}$ (the index $i$ belongs only to
the second subset). Thus $u_i<u_j$. The converse may be shown by
induction on $j-i$. For $j=i+1$, $\pi(i)>\pi(i+1)$ implies that the set 
 $\{k<i+1\::\: \pi(k)<\pi(i+1)\}$ is a subset of $\{k<i\::\:
\pi(k)<\pi(i)\}$, thus $u_i\geq u_{i+1}$. Therefore $u_i<u_{i+1}$
implies $\pi(i)<\pi(i+1)$. Assume now that $u_i\leq u_{i+1},\ldots, u_j$
holds and that we have already shown $\pi(i)<\pi(i+1),\ldots,
\pi(j-1)$. Assume, by way of contradiction, that $\pi(i)>\pi(j)$
holds. Then there is no $k$ satisfying $i<k<j$ and $\pi(k)<\pi(j)$ thus
$\{k<j\::\: \pi(k)<\pi(j)\}$ is a subset of $\{k<i\::\:
\pi(k)<\pi(i)\}$, implying $u_i\geq u_j$, a contradiction. Therefore  we
obtain $\pi(i)<\pi(j)$.   
\end{proof}
\begin{corollary}
Let $u_1\cdots u_n$ be the PNT of a permutation $\pi\in S_n$. Then 
$u_i\cdots u_j$ satisfies $u_i<u_{i+1},\ldots, u_{j-1}$ and $u_i\geq u_j$
if and only if $\pi(j)<\pi(i)< \pi(i+1),\ldots,\pi(j-1)$ holds.
\end{corollary}
\begin{corollary}
\label{C_fBperm}
Let $u_1\cdots u_n$ be the PNT of a permutation $\pi\in S_n$. Then
$u_1\cdots u_n$ is a kernel position in the flat Bernoulli game, if and
only if there exists a set of indices $1=i_0<i_1<\cdots<i_{k+1}=n$ such
that for each $j\in\{0,\ldots,k\}$ we have
$\pi(i_{j+1})<\pi(i_j+1)<\pi(i_j+2), \pi(i_j+3),\ldots,\pi(i_{j+1}-1)$.
\end{corollary}
Equation (\ref{E_fB}) also follows from Corollary~\ref{C_fBperm}.
In fact, there are $n$ ways to select $\pi(1)$. Then, introducing 
$m_j:=i_{j}-i_{j-1}-1$ for $j\geq 1$, we have
$\binom{n-1}{m_1,\ldots,m_k}$ ways to select the partitioning
$$
\{1,\ldots,n\}\setminus \pi(1)=\biguplus_{j=0}^k
\pi\left(\{i_j+1,\ldots,i_{j+1}\}\right) 
$$
and, for each $j$ there are $(i_{j+1}-i_j-2)!=(m_j-2)!$ ways to select
the partial permutation $\pi(i_j+1)\cdots\pi(i_{j+1})$. Both Equation
(\ref{E_fB}) and Corollary~\ref{C_fBperm} suggest looking at the numbers 
\begin{equation}
\label{E_K}
K_n=\kappa_{n+1}/(n+1)=
\sum_{k=1}^{\lfloor n/2\rfloor}
\sum_{\substack{m_1+\cdots + m_{k}=n\\ m_1,\ldots, m_k\geq 2}}
\binom{n}{m_1, \ldots, m_k} (m_1-2)!\cdots (m_k-2)!\quad\mbox{for $n\geq
  0$}. 
\end{equation}
It is easy to check the following statement.
\begin{proposition}
$K_n$ is the number of kernel positions of rank $n$ in the {\em exception-free}
  variant of the flat Bernoulli game, where removing the entire word if
  $u_1<u_2,\ldots, u_n$ is also a valid move, and the empty word is a
  valid position.
\end{proposition}
Corollary~\ref{C_fBperm} may be rephrased as follows.
\begin{corollary}
\label{C_fBpermK}
$K_n$ is the number of those permutations $\pi\in S_n$ for which 
there exists a set of indices $0=i_0<i_1<\cdots<i_{k+1}=n$ such
that for each $j\in\{0,\ldots,k\}$ we have
$\pi(i_{j+1})<\pi(i_j+1)<\pi(i_j+2), \pi(i_j+3),\ldots,\pi(i_{j+1}-1)$.
\end{corollary}
The generating function of the numbers $K_n$ is 
\begin{equation}
\label{E_Kgen}
\sum_{n=0}^{\infty} \frac{K_n}{n!} t^n
=\frac{1}{(1-t)(1-\ln(1-t))}.
\end{equation}
This formula may be derived not only from $K_n=\kappa_{n+1}/(n+1)$
and (\ref{E_fBgen}), but also from from Corollary~\ref{C_fBpermK} and the 
compositional formula for exponential generating
functions~\cite[Thm.\ 5.5.4]{Stanley-EC2}. We only need to observe that 
$$
\frac{1}{(1-t)(1-\ln(1-t))}=\frac{1}{1-t}\circ \left(t+(1-t)\ln(1-t)\right),
$$
where   
$$
\frac{1}{1-t}=\sum_{n=0}^{\infty} \frac{n! t^n}{n!}
$$
is the exponential generating function of linear orders, whereas
$$
t+(1-t)\ln(1-t)=-t\ln(1-t)-(-\ln(1-t-t))
=\sum_{n=1}^{\infty}\frac{t^{n+1}}{n}-\sum_{n=2}^{\infty}\frac{t^{n}}{n}
=\sum_{n=2}^{\infty}\frac{(n-2)! t^{n}}{n!}
$$
is the exponential generating function of linear orders of
$\{1,\ldots,n\}$, listing $1$ last and $2$ first.

By taking the antiderivative on both sides of (\ref{E_Kgen}) we obtain
$$
\sum_{n=0}^\infty \frac{K_n}{(n+1)!} t^{n+1}=\int \frac{1}{(1-t)(1-\ln(1-t))}\
dt=\ln(1-\ln(1-t))+K_{-1}.
$$
Introducing $K_{-1}:=0$, the numbers $K_{-1}, K_0,K_1,\ldots$ are listed
as sequence A089064 in the On-Line Encyclopedia of Integer
Sequences~\cite{OEIS}. There we may also find the formula 
\begin{equation}
\label{E_st}
K_n=(-1)^{n}\sum_{k=1}^{n+1}  s(n+1,k)\cdot (k-1)!
\end{equation}
expressing them in terms of the Stirling numbers of the first kind.
Using the well-known formulas 
$$
\sum_{k=1}^n s(n+1,k) x^k=x(x-1)\cdots (x-n)
\quad
\mbox{and} 
\quad
n!=\int_0^{\infty} x^n e^{-x}\ dx,
$$
Equation (\ref{E_st}) is equivalent to
\begin{equation}
\label{E_Kint}
K_n=(-1)^{n}\int_0^{\infty} (x-1)\cdots(x-n) e^{-x}\ dx.
\end{equation}
This formula may be directly verified by substituting it into the left
hand side of (\ref{E_Kgen}) and obtaining
$$
\int_{0}^{\infty} e^{-x}
\sum_{n=0}^{\infty}\binom{x-1}{n}(-t)^n\ dx
=\int_{0}^{\infty} e^{-x} (1-t)^{x-1}\ dx
=\frac{1}{(1-t)(1-\ln(1-t))}.
$$
We conclude this section with an intriguing conjecture.
By inspection of (\ref{E_fBgen}) and (\ref{E_Kgen}) we obtain the
following formula. 
\begin{lemma}
For $n\geq 1$, 
$$a_{n}:=(-1)^{n}\frac{(\kappa_{n+1}-(n+1)\cdot \kappa_{n})}{n+1}
=(-1)^{n}\left(K_n-n\cdot K_{n-1}\right)
$$
is the coefficient of $t^{n}/n!$ in $1/(1-\ln(1+t))$. 
\end{lemma}
The numbers $a_0,a_1,\ldots $ are listed as sequence A006252 in the
On-Line Encyclopedia of Integer Sequences~\cite{OEIS}. The first $11$
entries are positive, then $a_{12}=-519312$ is negative, the subsequent
entries seem to have alternating signs. The conjecture that this
alternation continues indefinitely, may be rephrased as follows.
\begin{conjecture}
\label{C_novice}
For $n\geq 12$ we have $n\cdot\kappa_{n-1}> \kappa_n$.  
Equivalently, $n\cdot K_{n-1}>K_n$ holds for $n\geq 11$.
\end{conjecture}
We may call Conjecture~\ref{C_novice} the {\em novice's chance}. Imagine
that the first player asks a novice friend to replace him or her for
just the first move in a flat Bernoulli game starting from a random
position of rank $n\geq 12$. If Conjecture~\ref{C_novice} is correct
then novice could simply remove the last letter,
because the number of nonkernel positions in which this is the first
move of the first player's winning strategy still exceeds the
number of all kernel positions. We should note that for the original
Bernoulli game a novice has no such chance. In that game the removal of
a single letter at the end of both words is not always a valid move, 
but we could advise our novice to remove the last letters at the end
of both words if this is a valid move and make a random valid move
otherwise. Our novice would have a chance if 
$$
\kappa_{n-1}\cdot \left(n^2-\binom{n-1}{2}\right)
=\kappa_{n-1}\cdot \binom{n+1}{2}\geq \kappa_n
$$
was true for all large $n$. However, it is known~\cite[\S 93]{Jordan} that
we have 
\begin{equation}
\label{E_b2bound}
\frac{n-2}{n}\left|\frac{b_{n-1}}{(n-1)!}\right|
<\left|\frac{b_{n}}{n!}\right|
<
\frac{n-1}{n}\left|\frac{b_{n-1}}{(n-1)!}\right|,\quad\mbox{implying}
\end{equation}
$$
(n-2)(n+1)\kappa_{n-1}<\kappa_n<(n-1)(n+1)\kappa_{n-1}.
$$
On the page of A006252 in~\cite{OEIS} we find that the
coefficient of $t^{n}/n!$ in $1/(1-\ln(1+t))$ is
\begin{equation}
\label{E_novicest}
\frac{(-1)^{n}(\kappa_{n+1}-(n+1)\cdot \kappa_{n})}{n+1}
=(-1)^{n}\left(K_n-n\cdot K_{n-1}\right)=\sum_{k=0}^{n} s(n,k) k!
\end{equation}
Equivalently, 
\begin{equation}
\label{E_noviceint}
\frac{(-1)^{n}(\kappa_{n+1}-(n+1)\cdot \kappa_{n})}{n+1}=
(-1)^{n}\left(K_n-n\cdot K_{n-1}\right)=\int_0^{\infty}
x(x-1)\cdots (x-n+1) e^{-x}\ dx.
\end{equation}
Equations (\ref{E_novicest}) and (\ref{E_noviceint}) may be verified the
same way as the analogous formulas (\ref{E_st}) and (\ref{E_Kint}). 
Therefore we may rewrite Conjecture~\ref{C_novice} as
follows:
\begin{equation}
\label{E_novice}
(-1)^{n}\int_0^{\infty} x(x-1)\cdots(x-n+1) e^{-x}\ dx
>0\quad\mbox{holds for $n\geq 11$.}
\end{equation}
This form indicates well the complication that arises, compared to the
original Bernoulli game. To prove (\ref{E_b2bound}),
Jordan~\cite[\S 93]{Jordan} uses the formula 
$$
\frac{b_n}{n!}=\int_0^1 \binom{x}{n} \ dx
$$
and is able to use the mean value theorem to compare $b_n/n!$ with
$b_{n+1}/(n+1)!$, because the function 
$\binom{x}{n}$ does not change sign on the interval $(0,1)$. Proving Equation 
(\ref{E_novice}) is equivalent to a similar estimate of the change of
the integral $(-1)^n\int_0^{\infty} (x-1)\cdots(x-n) e^{-x}\ dx$ as we
increase $n$, however, this integrand does change the sign several times
on the interval $(0,\infty)$. 

\section{Concluding remarks}
\label{s_c}

Conjecture~\ref{C_novice}, if true, would be an intriguing
example of a sequence ``finding its correct signature pattern'' after a
relatively long ``exceptional initial segment''. Many such examples seem
to exist in analysis, and it is perhaps time for combinatorialists to start
developing a method of proving some of them. 

Some of the most interesting questions
arising in connection with this paper seem to be related to the
instant Bernoulli game, presented in Section~\ref{s_MR}. The fact
that our decomposition into elementary kernel factors is
bijectively equivalent to King's~\cite{King} construction 
raises the suspicion that this decomposition may
also have an algebraic importance beyond the combinatorial one. This
suspicion is underscored by the fact that the correspondence between our
decomposition and King's is via some modified inversion table, whereas
Aguiar and Sottile~\cite{Aguiar-Sottile} highlight the importance of the weak
order to the structure of the Malvenuto-Reutenauer Hopf algebra, first
pointed out by Loday and Ronco~\cite{Loday-Ronco}. The weak order is
based on comparing the sets of inversions of two
permutations. Depending the way we choose the basis of the self-dual
Malvenuto-Reutenauer Hopf algebra, expressing one of the product and
coproduct seems easy in terms of place-based non-inversion tables,
whereas the other seems very difficult. If we choose the representation
considered by Poirier and Reutenauer~\cite{Poirier-Reutenauer} where
connected permutations form the free algebra basis, then the product of
two permutations is easily expressed in terms of PNTs, thus the
elementary kernel factor decomposition might indicate the presence of a
larger algebra ``looming on the horizon'' in which the multiplicative
indecomposables of the Malvenuto-Reutenauer Hopf algebra become decomposable. 

We should also mention that the decomposition that is equivalent to the
removal of the last elementary kernel factor is only the first phase in
King's construction~\cite{King}, a lot of hard work is done
afterwards to find the transposition Gray code, while recursing
on these reduction steps. Our presentation allows to better visualize King's
entire ``rough'' decomposition ``at once'' and thus may be suitable to
attack the open question of finding an adjacent transposition Gray code.

Finally, the degenerate Bernoulli game indexed with
$(p,q)$~\cite[\S 6]{Hetyei-EKP} can also be shown to be isomorphic to a
strongly Bernoulli type truncation game. For this game, the number of kernel
positions of rank $n$ is $(-q)^n (n+1)!\beta_n(p/q,0)$~\cite[Thm.\
  6.2]{Hetyei-EKP}, where $\beta_n(p/q)$ is a degenerate Bernoulli
number. We leave the detailed analysis of this game 
to a future occasion.

\section*{Acknowledgements} This work was completed while the author was
on reassignment of duties sponsored by the University of North Carolina
at Charlotte. The author wishes to thank two anonymous referees for
helping substantially improve both the presentation and the contents of
this paper and Christian Krattenthaler for remembering the
exercise in Stanley's book~\cite[Ch.\ 1, Exercise 32b]{Stanley-EC1} on
strong fixed points.

\end{document}